\documentclass[aoas]{imsart}

\RequirePackage{amsthm,amsmath,amsfonts,amssymb}
\RequirePackage[authoryear]{natbib}
\RequirePackage[colorlinks,citecolor=blue,urlcolor=blue]{hyperref}
\RequirePackage{graphicx}
\usepackage[clean]{revdiff} 
\startlocaldefs
\theoremstyle{plain}
\newtheorem{theorem}{Theorem}[section]
\newtheorem{lemma}[theorem]{Lemma}
\newtheorem{proposition}[theorem]{Proposition} 

\theoremstyle{remark}
\newtheorem{definition}[theorem]{Definition}
\newtheorem{remark}[theorem]{Remark}
\newtheorem{assumption}[theorem]{Assumption}
\newtheorem{example}{Example}

\newcommand{\norm}[1]{\left\|#1 \right\|}
\def\E{\mathbb E}
\def\R{\mathbb R}
\def\tr{{\rm tr}}
\newcommand{\dd}{{\rm d}}
\newcommand{\lb}{\left}
\newcommand{\rb}{\right}

\DeclareMathOperator*{\argmax}{arg\,max}
\definecolor{newcolor}{rgb}{0,0.50,0}

\endlocaldefs

\begin{document}

\begin{frontmatter}
\title{Stability estimates for the expected utility in Bayesian optimal experimental design}
\runtitle{Stability for Bayesian optimal experimental design}

\begin{aug}
\author[A]{\fnms{Duc-Lam}~\snm{Duong}\ead[label=e1]{duc-lam.duong@lut.fi}}\thanks{{\it Address for correspondence:} Department of Computational Engineering, School of Engineering Science, LUT University, Yliopistonkatu 34, 53850 Lappeenranta, Finland.},
\author[A]{\fnms{Tapio}~\snm{Helin}\ead[label=e2]{tapio.helin@lut.fi}}
\and
\author[A]{\fnms{Jose Rodrigo}~\snm{Rojo-Garcia}\ead[label=e3]{rodrigo.rojo.garcia@lut.fi}}


\address[A]{Computational Engineering, LUT School of Engineering Science, Lappeenranta-Lahti University of Technology, Finland\printead[presep={,\ }]{e1,e2,e3}}
\end{aug}

\begin{abstract}
We study stability properties of the expected utility function in Bayesian optimal experimental design. We provide a framework for this problem in a non-parametric setting and prove a convergence rate of the expected utility with respect to a likelihood perturbation. This rate is uniform over the design space and its sharpness in the general setting is demonstrated by proving a lower bound in a special case. To make the problem more concrete we proceed by considering non-linear Bayesian inverse problems with Gaussian likelihood and prove that the assumptions set out for the general case are satisfied and regain the stability of the expected utility with respect to perturbations to the observation map. Theoretical convergence rates are demonstrated numerically in three different examples.
\end{abstract}


\end{frontmatter}


\section{Introduction}

Acquisition of high quality data is the crux of many 
challenges in science and engineering. An outstanding example is the parameter estimation problem in statistical models. Namely, data collection, whether in field experiments or in laboratory, is often restricted by limited resources. It can be difficult, expensive and time-consuming, which puts severe limits on the quality of data acquired. To maximize the value of data for inference and minimize the uncertainty of the estimated parameters, one has to design the experiment (for instance, placing the sensors) in a way that is as economical and efficient as possible. This involves choosing the values of the controllable variables before the experiment takes place. Carefully designed experiments can make a substantial difference in accomplishing the tasks in an appropriate manner. \emph{Optimal experimental design} (OED) is a mathematical framework where a set of design variables with certain optimal criteria (based on the information matrix derived from the model) are satisfied (\cite{SH84, Puk06}). The most common criteria for OED include A-optimality and D-optimality which, in finite dimensions, seeks to minimize the trace and the determinant of the Fisher information matrix, respectively.

We adopt a Bayesian approach to OED \citep{CV95} that formulates the task as a maximization of an expected utility. Suppose $X$ denotes our unknown parameter, $Y$ stands for the observation and $d$ the design parameter. The expected utility $U$ is given by
\begin{equation}
    \label{eq:main_utility}
	U(d) = \E^\mu u(X,Y; d),
\end{equation}
where $u(X,Y; d)$ denotes the utility of an estimate $X$ given observation $Y$ with design $d$, and the expectation is taken w.r.t. the joint distribution $\mu$ of $X$ and $Y$.
The task of optimizing $U(d)$ is notoriously expensive especially if the design space is large or $X$ and $Y$ are modelled in high-dimensional space \citep{RDMP16}.

The guiding set of design problems for this paper is those emerging in modern inverse problems \citep{engl1996regularization}, which often involve imaging a high-dimensional object by indirect observations. Bayesian approach to inverse problems has gained considerable interest during the last two decades (\cite{kaipio2006statistical, Stu10, DS17}). The underlying mathematical model in inverse problems is often governed by partial differential equations (PDEs) giving rise to complex high-dimensional likelihood models.


Exploring the complex high-dimensional posterior distributions in Bayesian inversion is computationally costly but has become standard in literature in recent years due to developments in Monte Carlo (MC) based sampling schemes and the growth of available computational resources. To accelerate this task, various approximation schemes such as surrogate modelling of the likelihood distribution can also be applied including e.g. polynomial chaos expansion \citep{marzouk2007stochastic, schillings2013sparse} or neural networks \citep{herrmann2020deep}. Moreover, stability of the posterior distribution with respect to such approximations is well-understood (see \cite{sprungk2020local, garbuno2023bayesian} and reference therein). 




In the OED framework, the computational effort compared to conventional Bayesian inversion is significantly larger due to the double expectation and optimization task. In consequence, approximation schemes have a key role in producing optimal designs in that framework.
However, questions about the stability of the expected utility have received limited attention. To the best of our knowledge, it has been only considered in terms of a fixed design, i.e., pointwise convergence of the expected utility. Tempone and others in \citep{beck2018fast, beck2020multilevel, long2015laplace} have developed approximation results for nested MC methods with and without utilising an additional Laplace approximation and analyse the optimal parametrisation of the method with respect to the approximation error versus computational effort. Since MC approximation is random, any such error bound is expressed in probabilistic terms. In particular, \cite{beck2018fast} provides a recipe to achieve a given probabilistic error tolerance with optimal computational effort.
In another recent line of work, Foster and others \citep{foster2019variational} develop error analysis for variational methods being applied in combination with MC methods and optimize the depth of variational approximation to achieve a rate of convergence ${\mathcal O}((N+K)^{-\frac 12})$ for $N$ samples from MC and $K$ optimization steps for the variational model. 

For Bayesian OED tasks involving optimization on a continuous design manifold (as is often the case in inverse problems), pointwise convergence does not provide the full picture of the stability in terms of the optimization task.
Instead, uniform approximation rates of given numerical schemes in a neighbourhood around the optimal design are preferable in that regard.
In this work, we study the uniform stability of the expected utility in Bayesian OED systematically where changes in likelihood function or observation map can be observed. 


Non-parametric inference in Bayesian inverse problems and OED is motivated by 
the conviction to leave discretization until the last possible moment \citep{Stu10}, hence giving rise to opportunities to choose appropriate and robust discretization methods. Non-parametric approach for OED in Bayesian inverse problems has been formalized by Alexanderian (see \cite{alexanderian2021optimal} and references therein). Let us note that the infinite-dimensional setting arises naturally in numerous works involving Bayesian OED in inverse problems constrained by PDEs (\cite{alexanderian2014optimal, long2015fast, alexanderian2016fast, alexanderian2016bayesian, beck2018fast, wu2020fast}, to name a few), integral geometry \citep{haber2008numerical, ruthotto2018optimal, burger2021sequentially, helin2022edge} or nonlinear systems \citep{huan2013simulation, huan2014gradient}.

\subsection{Our contribution}
This work contributes to the rigorous study of the mathematical framework of Bayesian OED. 

\begin{itemize}
    \item We formulate the OED stability problem under a Bayesian framework in a non-parametric setting.
    We propose a set of assumptions (along the lines of \cite{Stu10}) under which the stability problem for OED can be addressed in a systematic way (Assumption \ref{ass:1}). In particular, we assume that the likelihood and its approximate version are close in the Kullback--Leibler divergence.
    \item We establish the convergence of the expected utility for a general approximation scheme satisfying Assumption \ref{ass:1} with a rate of one-half of the likelihood convergence rate (see Theorem \ref{thm:main1}). We demonstrate by a trivial example that a faster rate is not possible without further assumptions (see Example \ref{example:optimality}). 
    \item Together with the convergence of the surrogate expected utility, we prove that their maximizers converge, up to a subsequence, to a maximizer of the true expected utility (see Theorem \ref{thm:main2}). This ensures that the optimal design variable is also stable in the approximation. 
    \item As an important application, we consider some Bayesian inverse problems with Gaussian noise and their observation map can be replaced by some surrogate model.  We demonstrate that the assumptions we set out previously are satisfied in this case, given that the observation map and its surrogate model are close in certain norms (see Proposition \ref{prop:Gaussian} and Theorem \ref{thm:Gaussian}).
    \item Finally, we carry out numerical simulations on three different design problems. We observe that the rates predicted by our main theorems are aligned with the numerical results.
\end{itemize}

\subsection{Structure of the paper}
The paper is organised as follows. In Section \ref{sec:preliminaries}, we give an overview of the main objects of this article, including basic notions of non-parametric Bayesian inverse problems and Bayesian experimental design. We also summarize some background in probability measure theory, commonly used metrics between measures and introduce notations that will be used throughout this paper. In Section \ref{sec:stability}, we first outline the common framework including the general assumptions and main results. We proceed by establishing several lemmas and proving our main theorems. An important aspect of the main results is considered in Section \ref{sec:example} where we study the stability of OED for some Bayesian inverse problems with Gaussian noise.  In the last section, we provide three numerical examples to illustrate the results of the paper.

\section{Preliminaries}
\label{sec:preliminaries}

\subsection{Probability measures and metrics between measures}
Throughout this paper, $\mathcal{X}$ will be a separable Banach space (Hilbert space) equipped with a norm $\norm{\cdot}_\mathcal{X}$ (inner product $\langle \cdot \, , \,\cdot \rangle_{\mathcal{X}}$), with notice that the subscript may be ignored if no confusion arises. A bounded linear operator $\mathcal{C}: \mathcal{X} \to \mathcal{X}$ in Hilbert space $\mathcal{X}$ is called self-adjoint if $ \langle \mathcal{C}x , y \rangle = \langle x , \mathcal{C} y \rangle$ for all $x,y \in \mathcal{X}$ and positive definite (or positive) if $\langle \mathcal{C}x , x \rangle \ge 0$ for all $x\in \mathcal{X}$. We say that a self-adjoint and positive operator $\mathcal{C}$ is of trace class if
\begin{equation*}
    \tr({\mathcal{C}}) := \sum_{n = 1}^{\infty} \langle \mathcal{C} e_n, e_n\rangle < \infty,
\end{equation*}
where $\{e_n\}$ is an orthonormal basis of $\mathcal{X}$.  

Let $\mathcal{B}(\mathcal{X})$ be the Borel $\sigma-$algebra on $\mathcal{X}$ and let $\mu$ be a Borel probability measure on $\mathcal{X}$, that is, $\mu$ is defined on the measurable space $(\mathcal{X},\mathcal{B}(\mathcal{X})$. The mean $m\in \mathcal{X}$ and the (bounded linear operator) covariance $\mathcal{C}:\mathcal{X} \to \mathcal{X}$ of $\mu$ are defined as follows
\begin{equation*}
    \langle m, h \rangle = \int_{\mathcal{X}} \langle h, x \rangle \dd \mu(x), \quad \text{ for all } h \in \mathcal{X},
\end{equation*}
\begin{equation*}
    \langle \mathcal{C}h_1, h_2 \rangle = \int_{\mathcal{X}} \langle h_1, x - m \rangle \langle h_2, x - m_2 \rangle \mu(\dd x), \quad \text{ for all } h_1, h_2 \in \mathcal{X}.
\end{equation*}
We also use the concept of weighted norm $\|\cdot\|_{\mathcal{C}} = \|\mathcal{C}^{-1/2}\cdot\|$ for any covariance operator $\mathcal{C}$ in $\mathcal{X}$. 

Let $\mu_1$ and $\mu_2$ be two Borel probability measures on $\mathcal{X}$\rold{ and assume that $\mu_{2} \ll \mu_{1}$, i.e., $\mu_2$ is absolutely continuous w.r.t $\mu_1$}. Let $\mu$ be a common reference measure, also defined on $(\mathcal{X},\mathcal{B}(\mathcal{X})$\rold{, such that both $\mu_1$ and $\mu_2$ are absolutely continuous w.r.t $\mu$}. The following \rnew{``}distances\rnew{''} between measures are utilized in the rest of the paper.
 
The \emph{Hellinger distance} between $\mu_1$ and $\mu_2$ is defined as
\begin{equation*}
	d_{\mathrm{Hell}}^2\left(\mu_{1},\mu_{2}\right) =\frac{1}{2} \int_{\mathcal{X}}\left(\sqrt{\frac{\dd \mu_{1}}{\dd \mu}}-\sqrt{\frac{\dd \mu_{2}}{\dd \mu}}\right)^{2} \dd \mu
	= \frac{1}{2} \int_{\mathcal{X}}\left(1-\sqrt{\frac{\dd\mu_{2}}{\dd \mu_{1}}}\right)^{2} \dd \mu_{1},
\end{equation*}
where $\mu$ is a reference measure such that $\mu_1 \ll \mu$ and $\mu_2 \ll \mu$, i.e., $\mu_1$ and $\mu_2$ are absolutely continuous with respect to $\mu$. The second identity holds if $\mu_2 \ll \mu_1$.

The \emph{Kullback--Leibler} (KL) divergence for $\mu_1$ and $\mu_2$ with $\mu_2 \ll \mu_1$ is defined as
\begin{equation*}
    D_{\mathrm{KL}}\left(\mu_{2} \parallel \mu_{1}\right)=\int_{\mathcal{X}} \log \left(\frac{\dd \mu_{2}}{\dd \mu_{1}}\right)  \frac{\dd \mu_{2}}{\dd \mu_{1}} \dd \mu_{1}
    =\int_{\mathcal{X}} \log \left(\frac{\dd \mu_{2}}{\dd \mu_{1}}\right)  \dd \mu_{2}.
\end{equation*}
Here, the second identity holds if also $\mu_{1} \ll \mu_{2}$\rnew{ (that is, $\mu_1$ and $\mu_2$ are equivalent)}. 
Notice carefully that $D_{\mathrm{KL}}\left(\mu_{2} \parallel \mu_{1}\right) \ge 0$ and that KL divergence is not symmetric.

Let us now record three well-known lemmas that will be used below.
First, the Hellinger distance can be bounded by the Kullback--Leibler divergence as follows. 
\begin{lemma}
\label{lem:Hell-KL}
If $\mu_1$ and $\mu_2$ are equivalent probability measures on $\mathcal{X}$, then 
\begin{equation}
d^{2}_{\mathrm{Hell}}\left(\mu_{1}, \mu_{2}\right) \leq \frac{1}{2} D_{\mathrm{KL}}\left(\mu_{1} \parallel \mu_{2}\right).
\end{equation} 
\end{lemma}


Second, the Kullback--Leibler divergence between two Gaussian distributions has an explicit expression utilizing the means and covariance matrices.

\begin{lemma}
\label{lem:KL_Gaussian}
Suppose we have two Gaussian distributions $\mu_1 \sim {\mathcal N}(m_1, \Gamma_1)$ and $\mu_2 \sim {\mathcal N}(m_2, \Gamma_2)$ on $\R^p$. Then it holds that
\begin{equation*}
    D_{\mathrm{KL}}(\mu_1, \mu_2) = \frac 12\left(\tr(\Gamma_2^{-1} \Gamma_1) - p + (m_2 - m_1)^\top \Gamma_2^{-1} (m_2 - m_1) + \log\left(\frac{\det \Gamma_2}{\det \Gamma_1}\right)\right).
\end{equation*}
\end{lemma}

Third, arbitrary moments of Gaussian distributions in Hilbert spaces are finite.

\begin{lemma}[{\citet[Prop. 2.19]{DZ14}}]
\label{lem:moments}
Let $\mathcal{X}$ be a separable Hilbert space.  For any $k\in \mathbb{N}$, there exists a constant $C = C(k)$ such that
    \begin{equation*}
        \int_{\mathcal{X}} \norm{x}^{2k} \mu (\dd x) \le C[\tr(\Gamma)],
    \end{equation*}
for any Gaussian measure $\mu = \mathcal{N}(0, \Gamma)$.
\end{lemma}
Below, we denote random variables with capital letters ($X$ and $Y$), while the realizations are generally denoted by lowercase letters ($x$ and $y$).

\subsection{Bayesian optimal experimental design}

In Bayesian optimal experimental design, one seeks an experimental setup providing predictive data distribution with maximal information content in terms of recovering the unknown parameter. 
Let us make this idea more precise: we denote the unknown latent parameter by $x\in \mathcal{X}$, where $\mathcal{X}$ is a separable Banach space. Let $y\in \mathcal{Y}$ be the observational data variable   where $\mathcal{Y}$ is a finite-dimensional data space. For convenience, in this work we assume $\mathcal{Y} = \R^p$. Moreover, let $d \in \mathcal{D}$ be the design variable, where $\mathcal{D}$ is a (typically compact) metric space. 

The expected utility is given by formula \eqref{eq:main_utility},
where $\mu$ is the joint probability distribution of $X$ and $Y$, and $u$ is a utility function. In this work, we focus on the \emph{expected information gain} by defining
\begin{equation*}
	u(x,y,d) = -\log\left(\frac{\pi(y | x; d)}{\pi(y ; d)}\right),
\end{equation*}
where $\pi(y | x; d)$ and $\pi(y ; d)$ are the likelihood distribution and the evidence distribution given $d\in {\mathcal D}$, respectively. We observe that
\begin{equation}
	\label{eq:expected_util}
        U(d) = \int_{\mathcal{Y}}\int_{\mathcal{X}}\log \left(\dfrac{\pi(y\vert x;d)}{\pi(y;d)}\right) \pi(y\vert x;d) \dd  \mu_{0}(x) \dd y = \int_{{\mathcal X}} D_{KL}(\pi(\cdot | x; d), \pi(\cdot; d)) \mu_0(dx),
\end{equation}
where $\mu_0$ is the prior measure on $\mathcal{X}$.
To find the optimal experimental design, the expected utility is then maximized over the design variable space $\mathcal{D}$. A design $d^*$ is called optimal if it maximizes $U$, that is
\begin{equation} \label{in:d_optimal}
   d^* \in \argmax_{d\in \mathcal{D}} U(d). 
\end{equation}
We note that in general, the functional $U: \mathcal{D} \to \mathbb{R}$ may have several \rchange{maximisers}{maximizers}.


In inverse problems, the unknown $x$ is connected to the data $y$
through an observation (or parameter-to-observable) map $\mathcal{G}: \mathcal{X} \times {\mathcal D}\to \mathcal{Y}$.
The problem of inverting ${\mathcal G}(\cdot; d)$ with fixed $d\in {\mathcal D}$ is ill-posed and the likelihood distribution is governed by the observational model 
\begin{equation*}
    y = \mathcal{G}(x; d) + \xi,
\end{equation*}
where $\xi$ represents an additive measurement noise.
As a typical example, a Gaussian distribution noise distribution, $\xi \sim \mathcal{N}(0,\Gamma)$, with the covariance matrix $\Gamma \in \mathbb{R}^{p \times p}$ giving rise to the likelihood distribution $y|x \sim \mathcal{N}(\mathcal{G}(x),\Gamma)$
%

\subsection{$\Gamma$-convergence}
We collect here the definition and some basic results of $\Gamma$-convergence that will be used later on. Standard references of this subject are \cite{braides2002gamma, dal1993introduction}. 
\begin{definition} 
\label{def:Gamma_convergence}
    Let $\mathcal{X}$ be a metric space and assume that $F_n, F: \mathcal{X} \to \mathbb{R}$ are functionals on $\mathcal{X}$. We say that $F_n$ $\Gamma$-converges to $F$ if, for every $x\in \mathcal{X}$, the following conditions hold,
    \begin{itemize}
        \item [(i)] (liminf inequality) for every sequence $(x_n)_{n\in \mathbb{N}} \subset \mathcal{X}$ converging to $x$,
        \begin{equation*}
            F(x) \le \liminf_{n\to \infty} F_n(x_n);
        \end{equation*}
        \item [(ii)] (limsup inequality) there exists a recovery sequence $(x_n)_{n\in \mathbb{N}} \subset \mathcal{X}$ converging to $x$ such that
        \begin{equation*}
            F(x) \ge \limsup_{n\to \infty} F_n(x_n).
        \end{equation*}
    \end{itemize}
\end{definition}

\begin{theorem}[Fundamental Theorem of $\Gamma$-convergence]
\label{thm:Gamma_convergence}
    If $F_n$ $\Gamma$-converges to $F$ and $x_n$ minimizes $F_n$, then every limit point $x$ of the sequence $(x_n)_{n\in\mathbb{N}}$ is a minimizer of $F$, and
    \begin{equation*}
        F(x) = \limsup_{n\to\infty} F_n(x_n).
    \end{equation*}
\end{theorem}

\section{Stability estimates}
\label{sec:stability}
\subsection{General assumptions and main results}

In this section, we establish two useful results on the stability of the expected utility and of the optimal design variable. 
Recall now the expected information gain $U(d)$ defined in \eqref{eq:expected_util}.
Assume that we have access to a surrogate likelihood density $\pi_N(y | x; d)$, which approximates $\pi(y | x; d)$ as $N$ increases. The corresponding surrogate utility $U_N(d)$ is obtained as
\begin{equation}
	\label{eq:surrogate_util}
	U_{N}(d) = \int_{\mathcal{Y}}\int_{\mathcal{X}}\log \left(\dfrac{\pi_{N}(y|x;d)}{\pi_{N}(y;d)}\right) \pi_{N}(y|x;d) \dd \mu_{0}(x) \dd y,
\end{equation}
where $\pi_{N}(y;d)$ is the corresponding surrogate evidence. Let us make our approximation condition precise by the following assumption.
\begin{assumption}
\label{ass:1}
The following conditions hold.
\begin{itemize}
\item[\rm (A1)] There exist $C>0$, a function $\psi: \mathbb{N} \to \mathbb{R}_+$ such that $\psi(N) \to 0$ as $N\to \infty$ and
\begin{equation*}
\E^{\mu_0} \lb [D_{\mathrm{KL}} \lb( \pi_N(\cdot | X;d) \parallel \pi(\cdot | X;d) \rb) \rb] \leq C \psi(N),
\end{equation*}
for all $d\in {\mathcal D}$.
\item[\rm (A2)] For any converging sequence $d_N\to d$ in $\mathcal{D}$, we have
\begin{equation*}
    \lim_{N\to\infty} \E^{\mu_0} \lb [D_{\mathrm{KL}} \lb( \pi(\cdot | X;d_N) \parallel \pi(\cdot | X;d) \rb) \rb] = 0.
\end{equation*}

\item[\rm (A3)] There exists $C_0 > 0$ such that for every $N \in \mathbb{N}$, $d\in \mathcal{D}$ and every sequence $d_N \to d$ in $\mathcal{D}$,
\begin{equation} \label{eq:assA3}
\int_{\mathcal{Y}}\int_{\mathcal{X}} \log^2 \left(\frac{\pi(y|x;d)}{\pi(y;d)}\right) \left[\pi(y|x;d) +\pi_N(y|x;d) + \pi(y|x;d_N)\right]\, \dd \mu_0(x) \dd y < C_0.
\end{equation}
\end{itemize} 
\end{assumption}

\begin{remark}
    Assumption (A1) reflects a natural condition on the approximation rate of the surrogate likelihood under the Kullback--Leibler divergence. This is somewhat similar to the conditions commonly used in Bayesian inverse problems, albeit under different metrics (for instance, Hellinger distance, see \citet[\rnew{Section 4, }\rchange{a}{A}ssumptions 2]{DS17}). Assumption (A2) records a continuity condition of the likelihood with respect to the design variable. (A3) is a technical assumption which, in loose terms, requires that a special divergence between likelihood and data-marginal (or equivalently, between posterior and prior) is bounded when averaged over prior. We note that the second moment of the log-ratio quantity such as $\E^{\pi(\cdot|x;d)} \lb[\log^2 \left(\frac{\pi(Y|x;d)}{\pi(Y;d)}\right) \rb]$ in \eqref{eq:assA3} appears quite naturally in the context of Bayesian statistical inverse problems, see for instance \cite{nickl2022bayesian} (Proposition 1.3.1, where it is termed the $V$-distance). 
\end{remark}

The following proposition is the cornerstone of our main theorems. It yields in particular that the difference between the expected utility and its surrogate can be controlled by the likelihood perturbations under the Kullback--Leibler divergence. 

\begin{proposition}
\label{prop:main}
Consider likelihood distributions $\pi(y|x)$ and $\tilde \pi(y|x)$ and define
\begin{equation*}
	U = \int_{\mathcal{Y}}\int_{\mathcal{X}}\log \left(\dfrac{\pi(y | x)}{ \pi(y)}\right)  \pi(y|x)\dd \mu_{0}(x) \dd y
	\quad \text{and} \quad
	\widetilde{U} = \int_{\mathcal{Y}}\int_{\mathcal{X}}\log \left(\dfrac{\tilde \pi(y | x)}{\tilde \pi(y)}\right) \tilde \pi(y|x)\dd \mu_{0}(x) \dd y,
\end{equation*}
where $\pi(y) = \int \pi(y | x) \dd \mu_0 (x)$ and $\tilde \pi(y) = \int \tilde \pi(y | x) \dd \mu_0 (x)$. Let us denote
\begin{equation*}
K:= \int_{\mathcal{Y}}\int_{\mathcal{X}} \log^2 \left(\frac{\pi(y|x)}{\pi(y)}\right) \left[\pi(y|x) +\tilde \pi(y|x)\right]\, \dd \mu_0(x) \dd y.
\end{equation*}
It follows that
\begin{equation}
    |U - \widetilde U| \leq \sqrt{K}\sqrt{\E^{\mu_0} D_{\mathrm{KL}}(\pi(\cdot| X) \parallel \tilde \pi(\cdot| X))} + 2 \E^{\mu_0} D_{\mathrm{KL}}(\pi(\cdot| X) \parallel \tilde \pi(\cdot| X)). 
\end{equation}
\end{proposition}

We now state the main theorems of the paper.
\begin{theorem}
\label{thm:main1}
Let Assumption \ref{ass:1} (A1) and (A3) hold. Then there exists $C>0$ such that for all $N$ sufficiently large,
\begin{equation}
    \sup_{d\in \mathcal{D}} |U(d) - U_N(d)| \le C\sqrt{\psi(N)}.
\end{equation}
\end{theorem}

\begin{theorem}
\label{thm:main2}
Let Assumption \ref{ass:1} hold. Suppose
\begin{equation*}
   d^*_N \in \argmax_{d\in \mathcal{D}} U_N(d). 
\end{equation*}
Then, the limit $d^*$ of any converging subsequence of $\{d^*_N\}_{N=1}^\infty$ is a maximizer of $U$, that is, $d^* \in \argmax_{d\in \mathcal{D}} U(d)$. Moreover,
\begin{equation}\label{eq:exp_arg}
    \liminf_{N\to \infty} U_N(d^*_N) = U(d^*).
\end{equation}
In particular, if $\{d^*_N\}_{N=1}^\infty$ converges to $d^*$ in $\mathcal{D}$, then $d^*$ is a maximizer of $U$, and 
\begin{equation}\label{eq:maxvalue}
    \lim_{N\to \infty} U_N(d^*_N) = U(d^*).
\end{equation}
\end{theorem}

\begin{remark}
    Theorem \ref{thm:main1} establishes the uniform convergence of the approximate expected utility. Theorem \ref{thm:main2}, moreover, ensures that the corresponding approximate optimal design and the maximum expected information gain also converge up to a subsequence.
\end{remark}

\subsection{Proof of the main theorems}
\subsubsection{Proof of Proposition \ref{prop:main}}
Let us recall for the reader's convenience that
\begin{equation*}
	U = \int_{\mathcal{Y}}\int_{\mathcal{X}}\log \left(\dfrac{\pi(y | x)}{ \pi(y)}\right)  \pi(y|x)\dd \mu_{0}(x) \dd y
	\;\; \text{and} \;\;
	\widetilde{U} = \int_{\mathcal{Y}}\int_{\mathcal{X}}\log \left(\dfrac{\tilde \pi(y | x)}{\tilde \pi(y)}\right) \tilde \pi(y|x)\dd \mu_{0}(x) \dd y,
\end{equation*}
where
\begin{equation*}
    \pi(y) = \int \pi(y | x) \dd \mu_0 (x) \quad \text{and} \quad \tilde \pi(y) = \int \tilde \pi(y | x) \dd \mu_0 (x).
\end{equation*}
The corresponding posteriors are given by
\begin{equation} \label{eq:Bayes2}
    \frac{\dd \mu^y}{\dd \mu_0} (x) = \frac{\pi(y|x)}{\pi(y)}, \quad \frac{\dd \tilde\mu^y}{\dd \mu_0} (x) = \frac{\tilde\pi(y|x)}{\tilde\pi(y)}.
\end{equation}
We have
\begin{eqnarray*}
    U - \widetilde U & = &  \int_{\mathcal{Y}}\int_{\mathcal{X}} \left[\log \left(\dfrac{\pi(y\vert x)}{\pi(y)}\right) \pi(y\vert x) - \log\left(\frac{\tilde \pi(y | x)}{\tilde \pi(y)}\right) \tilde \pi(y | x)\right]\,\dd \mu_{0}(x) \dd y \\
        & = & \int_{\mathcal{Y}}\int_{\mathcal{X}} \log\left(\frac{\pi(y|x)}{\pi(y)}\right) \left[\pi(y|x) - \tilde \pi(y|x)\right]\, \dd \mu_{0}(x) \dd y \\
	& & + \int_{\mathcal{Y}} \log\left(\frac{\tilde \pi(y)}{\pi(y)}\right)\int_{\mathcal{X}} \tilde \pi(y|x)\, \dd \mu_{0}(x) \dd y \\ 
	& & +\int_{\mathcal{X}} \int_{\mathcal{Y}} \log\left(\frac{\pi(y|x)}{\tilde \pi(y|x)}\right) \tilde \pi(y|x) \, \dd \mu_{0}(x) \dd y \\ 
        & = & I + D_{\mathrm{KL}}(\tilde \pi(\cdot) \parallel \pi(\cdot)) - \E^{\mu_{0}} D_{\mathrm{KL}}(\tilde \pi(\cdot|X) \parallel \pi(\cdot|X)),
\end{eqnarray*}
where 
\begin{equation*}
I := \int_{\mathcal{Y}}\int_{\mathcal{X}} \log\left(\frac{\pi(y|x)}{\pi(y)}\right) \left[\pi(y|x) - \tilde \pi(y|x)\right]\, \dd \mu_{0}(x) \dd y.
\end{equation*}
Since we have the identity above, we naturally have
\begin{equation}
\label{eq:first_est}
|U - \widetilde U| \leq |I| + |D_{\mathrm{KL}}\lb(\tilde \pi(\cdot) \parallel \pi(\cdot)\rb)| + |\E^{\mu_{0}} D_{\mathrm{KL}}(\tilde \pi(\cdot|X) \parallel \pi(\cdot|X))|.
\end{equation}
Let us first consider the term $I$.
	
\begin{lemma} 
\label{lem:Iestimate}
It follows that
\begin{equation*}
|I|^2 \leq K \E^{\mu_{0}} \left[D_{\mathrm{KL}} (\tilde \pi(\cdot|X) \parallel \pi(\cdot|X)) \right],
\end{equation*}
where 
\begin{equation*}
K = \int_{\mathcal{Y}}\int_{\mathcal{X}}  \log^2 \left(\frac{\pi(y|x)}{\pi(y)}\right) \left[\pi(y|x) + \tilde \pi(y|x)\right]\, \dd \mu_{0}(x) \dd y.
\end{equation*}
\end{lemma}

\begin{proof}
By the Cauchy--Schwartz inequality,
\begin{eqnarray*}
I^2 & = & \left(\int_{\mathcal{Y}}\int_{\mathcal{X}} \log\left(\frac{\pi(y|x)}{\pi(y)}\right)\left[\sqrt{\pi(y|x)} + \sqrt{\tilde \pi(y|x)}\right] \left[\sqrt{\pi(y|x)} - \sqrt{\tilde \pi(y|x)}\right]\, \dd \mu_{0}(x) \dd y\right)^2 \\
& \le & \int_{\mathcal{Y}}\int_{\mathcal{X}}  \log^2 \left(\frac{\pi(y|x)}{\pi(y)}\right) \left[\sqrt{\pi(y|x)} + \sqrt{\tilde \pi(y|x)}\right]^2\, \dd \mu_{0}(x) \dd y \\
& & \cdot \int_{\mathcal{Y}}\int_{\mathcal{X}}  \left[\sqrt{\pi(y|x)} - \sqrt{\tilde \pi(y|x)}\right]^2 \, \dd \mu_{0}(x) \dd y\\
& \leq & 2 \int_{\mathcal{Y}}\int_{\mathcal{X}}  \log^2 \left(\frac{\pi(y|x)}{\pi(y)}\right) \left[\pi(y|x) + \tilde \pi(y|x)\right]\, \dd \mu_{0}(x) \dd y
\cdot \E^{\mu_{0}} \left[d_{\mathrm{Hell}}^2(\pi(\cdot |X), \tilde \pi(\cdot | X))\right] 
\end{eqnarray*}
Now, thanks to Lemma \ref{lem:Hell-KL},
\begin{eqnarray*}
\E^{\mu_{0}} \left[d_{\mathrm{Hell}}^2(\tilde \pi(\cdot |X), \pi(\cdot | X))\right]   \leq  \dfrac{1}{2} \E^{\mu_{0}} \left[D_{\mathrm{KL}} (\tilde \pi(\cdot|X)) \parallel \pi(\cdot|X)  \right].
\end{eqnarray*}
Therefore
\begin{equation*}
|I|^2 \leq K \E^{\mu_{0}} \left[D_{\mathrm{KL}} (\tilde \pi(\cdot|X)) \parallel \pi(\cdot|X)  \right],
\end{equation*}
as required.
\end{proof}

Let us now consider the second term on the right-hand side of \eqref{eq:first_est}.
	
\begin{lemma} 
\label{lem:convKL}
It holds that
\begin{equation*}
    D_{\mathrm{KL}}(\tilde \pi(\cdot) \parallel \pi(\cdot)) \leq \E^{\mu_{0}} \left[ D_{\mathrm{KL}}(\tilde \pi(\cdot|X) \parallel \pi(\cdot|X))\right].
\end{equation*}
\end{lemma}
\begin{proof}
Thanks to \eqref{eq:Bayes2} and Fubini's theorem, we have
\begin{eqnarray*}
	\E^{\mu_{0}} \left[ D_{\mathrm{KL}}(\tilde \pi(\cdot|X) \parallel \pi(\cdot|X))\right]  & = &  \int_{\mathcal{Y}}\int_{\mathcal{X}} \log\left(\frac{\tilde \pi(y|x)}{\pi(y|x)}\right) \tilde \pi(y|x)\, \dd \mu_{0}(x) \dd y\\
	& = &  \int_{\mathcal{Y}}\int_{\mathcal{X}} \log\left(\frac{\tilde \pi(y)}{\pi(y)}\right) \tilde\pi(y)\frac{\dd \tilde\mu^{y}}{\dd \mu_{0}}(x)\, \dd \mu_{0}(x) \dd y\\
	& & + \int_{\mathcal{Y}}\int_{\mathcal{X}} \log\left(\frac{\dd \tilde\mu^{y}}{\dd \mu^{y}}(x)\right) \tilde\pi(y)\frac{\dd \tilde\mu^{y}}{\dd \mu_{0}}(x)\, \dd \mu_{0}(x) \dd y\\	
	& = &  \int_{\mathcal{X}}\left[ \int_{\mathcal{Y}} \log\left(\frac{\tilde\pi(y)}{\pi(y)}\right) \tilde \pi(y)\, \dd y\right] \dd \tilde\mu^{y}(x)\\
	& & + \int_{\mathcal{Y}}\left[ \int_{\mathcal{X}} \log\left(\frac{\dd \tilde\mu^{y}}{\dd \mu^{y}}(x)\right) \dd \tilde\mu^{y}(x)\right]\,\tilde\pi(y) \dd y\\
	& = &  D_{\mathrm{KL}}(\tilde \pi(\cdot) \parallel \pi(\cdot))  +  \E^{\tilde\pi(\cdot)}\left[ D_{\mathrm{KL}}(\tilde\mu^{Y}(\cdot) \parallel \mu^{Y}(\cdot)\right]	
\end{eqnarray*} 
Since the Kullback--Leibler divergence is always nonnegative, we obtain
\begin{eqnarray*}
	D_{\mathrm{KL}}(\tilde \pi(\cdot) \parallel \pi(\cdot)) & = & \E^{\mu_{0}} \left[ D_{\mathrm{KL}}(\tilde \pi(\cdot|X) \parallel \pi(\cdot|X))\right] - \E^{\tilde\pi(\cdot)}\left[ D_{\mathrm{KL}}(\tilde\mu^{Y}(\cdot) \parallel \mu^{Y}(\cdot)\right] \\
	& \leq & \E^{\mu_{0}} \left[ D_{\mathrm{KL}}(\tilde \pi(\cdot|X) \parallel \pi(\cdot|X))\right],	
\end{eqnarray*}
which proves the claim.
\end{proof}

\begin{proof}[Proof of Proposition \ref{prop:main}]
    Proposition \ref{prop:main} follows immediately from \eqref{eq:first_est}, Lemma \ref{lem:Iestimate} and Lemma \ref{lem:convKL}.
\end{proof}

\subsubsection{Proof of Theorem \ref{thm:main1}}
Using Proposition \ref{prop:main} by making the dependence of the likelihood, evidence and expected information gain on the design variable explicit, and replacing $\tilde \pi(y|x), \tilde \pi(y)$ and $\widetilde U$ by $\pi_N(y|x;d), \pi_N(y;d)$ and $U_N(d)$, respectively, we have
\begin{eqnarray*}
|U(d) - U_N(d)| & \leq & \sqrt K_1 \sqrt{\E^{\mu_0} D_{\mathrm{KL}}(\pi_N(\cdot|X;d) \parallel \pi(\cdot|X;d))} \\ && + 2 |\E^{\mu_0} D_{\mathrm{KL}}(\pi_N(\cdot|X;d) \parallel \pi(\cdot|X;d))|,
\end{eqnarray*}
where 
\begin{equation*}
K_1:= \int_{\mathcal{Y}} \int_{\mathcal{X}} \log^2 \left(\frac{\pi(y|x;d)}{\pi(y;d)}\right) \left[\pi(y|x;d) + \pi_N(y|x;d)\right]\, \dd \mu_0(x) \dd y.
\end{equation*} 
Thanks to Assumption \ref{ass:1} (A1), $\E^{\mu_0} D_{\mathrm{KL}}(\pi_N(\cdot|X;d) \parallel \pi(\cdot|X;d)) \to 0$ as $N\to \infty$. Therefore, there exists $N_0>0$ large enough such that $|\E^{\mu_0} D_{\mathrm{KL}}(\pi_N(\cdot|X;d) \parallel \pi(\cdot|X;d))| < 1/4$ for $N> N_0$. Note that by Assumption \ref{ass:1} (A3), $K_1\le C_0$ with $C_0$ does not depend on $N$. Now choose $C = \sqrt{C_0} + 1$, we arrive at
\begin{equation*}
|U(d) - U_N(d)| \leq C \sqrt{\psi(N)},
\end{equation*}
for large $N$. Take supremum both sides over $d \in \mathcal{D}$ we conclude the proof of Theorem \ref{thm:main1}.

\subsubsection{Proof of Theorem \ref{thm:main2}}

Let $d\in \mathcal{D}$. Suppose $d_N\to d$ in $\mathcal{D}$. We have
\begin{eqnarray*}
    |U(d) - U_N(d_N)| & \leq & |U(d) - U(d_N)| + |U(d_N) - U_N(d_N)| \\
    & \leq & |U(d) - U(d_N)| + \sup_{d'\in D} |U(d') - U_N(d')|.
\end{eqnarray*}
By Proposition \ref{prop:main}, where we make the dependence of the likelihood, evidence and expected information gain on the design variable explicit, and with $\tilde \pi(y|x), \tilde \pi(y)$ and $\widetilde U$ being replaced by $\pi(y|x;d_N), \pi(y;d_N)$ and $U(d_N)$, respectively, it holds that
\begin{equation} \label{eq:Ucontinuity}
\begin{aligned}
    |U(d) - U(d_N)|  \le & \sqrt{K_2}\sqrt{\E^{\mu_0} D_{\mathrm{KL}}(\pi(\cdot|x; d_N) \parallel \pi(\cdot|x; d))} \\ & + \E^{\mu_0} D_{\mathrm{KL}}(\pi(\cdot|X; d_N) \parallel \pi(\cdot|X; d)), 
\end{aligned} 
\end{equation}
where 
\begin{equation*}
K_2:= \int_{\mathcal{Y}} \int_{\mathcal{X}} \log^2 \left(\frac{\pi(y|x;d)}{\pi(y;d)}\right) \left[\pi(y|x;d) + \pi(y|x;d_N)\right]\, \dd \mu_0(x) \dd y.
\end{equation*} 
It follows from Assumption \ref{ass:1} (A3) that there exists $C_0 > 0$, independent of $d$, such that $K_2 \le C_0$. 
This and \eqref{eq:Ucontinuity} imply, by Assumption \ref{ass:1} (A2),
\begin{equation*}
    \lim_{N\to \infty } |U(d) - U(d_N)| \to 0.
\end{equation*}
Now by Theorem \ref{thm:main1}, and note that $\psi(N) \to 0$ as $N\to \infty$,
\begin{equation}
\label{eq:uniform}
    \lim_{N\to \infty} \sup_{d'\in D} |U(d') - U_N(d')| \to 0.
\end{equation}
It follows that, for every $d_N \to d$,
\begin{equation}
\label{eq:continuous_conv}
    \lim_{N\to \infty} U_N(d_N) = U(d).
\end{equation}
Therefore, the liminf inequality in Definition \ref{def:Gamma_convergence} is satisfied (as an equality). The limsup inequality is trivial, since for every $d$ we can choose the constant sequence $d_N = d$, and it follows from \eqref{eq:uniform} that $\limsup_{N\to \infty} U_N(d_N) = \limsup_{N\to \infty} U_N(d) = U(d)$. Thus,
\begin{equation*}
    U_N \quad \Gamma\text{-converges to } \; U \; \text{ as } \; N\to\infty. 
\end{equation*}
The rest follows from the Fundamental Theorem of $\Gamma$-convergence (Theorem \ref{thm:Gamma_convergence}). This completes the proof.

\begin{remark}
It follows from \eqref{eq:continuous_conv} that $U_N$ \emph{continuously converges} to $U$ (see \citet[Definition 4.7]{dal1993introduction}), which is strictly stronger than $\Gamma$-convergence. In fact, this continuous convergence also implies that $-U_N$ $\Gamma$-converges to $-U$, which ensures (by symmetry) that the limit of any convergent maximizing sequences of $U_N$ is a maximizer of $U$. (Note carefully that $U_N$ $\Gamma$-converges to $U$ does \emph{not} imply $-U_N$ $\Gamma$-converges to $-U$, see \citet[Example 1.12]{braides2002gamma}).
\end{remark}

\section{Gaussian likelihood in Bayesian inverse problem}
\label{sec:example}
	
In this section, we consider the inverse problem
\begin{equation} \label{eq:inverse_example}
	y = {\mathcal G}(x;d) + \epsilon,
\end{equation}
for $x\in \mathcal{X}$ and $y, \epsilon \in \R^p$, where the noise has multivariate Gaussian distribution $\epsilon \sim {\mathcal N}(0, \Gamma)$ with some positive definite matrix $\Gamma$. Suppose we approximate the observation map ${\mathcal G}$ with a surrogate model ${\mathcal G}_N$. 
Now, thanks to Lemma \ref{lem:KL_Gaussian},
\begin{equation*}
    \E^{\mu_0} D_{\mathrm{KL}}(\pi_N(\cdot | X; d) \parallel \pi(\cdot | X; d)) = \frac{1}{2} \E^{\mu_0} \norm{{\mathcal G}(X;d) - {\mathcal G}_N(X;d)}_{\Gamma}^2
\end{equation*}
and 
\begin{equation}
	\label{eq:gauss_like_DKL2}
    \E^{\mu_0} D_{\mathrm{KL}}(\pi(\cdot | X; d_N)
    \parallel \pi(\cdot | X; d)) = \frac{1}{2} \E^{\mu_0} \norm{{\mathcal G}(X;d) - {\mathcal G}(X;d_N)}_{\Gamma}^2
\end{equation}
giving a natural interpretation to (A1) and (A2) in Assumption \ref{ass:1} in terms of convergence in $\Gamma$-weighted $L^2(\mu_0)$, i.e., ${\mathcal G}_N(X;d)$ should converge to ${\mathcal G}(X;d)$ uniformly in $d$, while ${\mathcal G}(X;d)$ is required to be $L^2(\mu_0)$-continuous with respect to $d$.

Let us next make the following assumption on ${\mathcal G}$ and ${\mathcal G}_N$. 

\begin{assumption}\label{ass:2} 
The observation operator and its surrogate version are bounded in $\Gamma$-weighted $L^4(\mu_0)$ uniformly in $d$, that is, there exists constant $C_G>0$ and  such that
\begin{equation*}
	\sup_{d\in \mathcal{D}} \int_{{\mathcal X}}\norm{{\mathcal G}(x; d)}_{\Gamma}^4 \,\dd \mu_0(x) \le C_G, \quad \sup_{d\in \mathcal{D}} \int_{{\mathcal X}}\norm{\mathcal{G}_N(x;d)}_{\Gamma}^4 \,\dd \mu_0(x) \le C_G
\end{equation*}
for all $N\in \mathbb{N}$.
\end{assumption}

It turns out that  the third condition (A3) in Assumption \ref{ass:1} is implied by Assumption \ref{ass:2}.
	
\begin{proposition} \label{prop:Gaussian}
Let Assumption \ref{ass:2} hold. It follows that, for Gaussian noise $\epsilon \sim {\mathcal N}(0, \Gamma)$, there exists a constant $C$ depending only on $\tr(\Gamma)$ and $C_G$ such that
\begin{equation}
\label{eq:Prop_Kfinite}
    K := \int_{\mathbb{R}^p} \int_{\mathcal{X}} \log^2 \left(\frac{\pi(y|x;d)}{\pi(y;d)}\right) \left[\pi(y|x;d) +\pi_N(y|x;d))\right]\, \dd \mu_0(x) \dd y \le C.
\end{equation}
\end{proposition}
	
\begin{proof}
We rewrite $K$ as a sum $K=K_1 + K_2$, where
\begin{equation*}
    K_1 = \int_{\mathbb{R}^p} \int_{\mathcal{X}} \log^2 \left(\frac{\pi(y|x;d)}{\pi(y;d)}\right) \pi(y|x;d) \, \dd \mu_0(x) \dd y,
\end{equation*}
and 
\begin{equation*}
    K_2 = \int_{\mathbb{R}^p} \int_{\mathcal{X}} \log^2 \left(\frac{\pi(y|x;d)}{\pi(y;d)}\right) \pi_N(y|x;d) \, \dd \mu_0(x) \dd y.
\end{equation*}
By Cauchy inequality $(a+b)^2 \le 2(a^2 + b^2)$,
\begin{eqnarray*}
	\log^2 \left(\frac{\pi(y|x;d)}{\pi(y;d)} \right) 
	& = & \log^2 \left(\dfrac{\exp\left(-\frac{1}{2} \norm{y- {\mathcal G}(x;d)}_{\Gamma}^2\right)}{T\E^{\mu_0}\pi(y|X;d)}\right) \\
	& = & \left( -\frac{1}{2} \norm{y- {\mathcal G}(x; d)}_{\Gamma}^2 - \log \left(T\E^{\mu_0}\pi(y|X;d)\right) \right)^2\\
	& \le & \frac{1}{2}\norm{y- {\mathcal G}(x; d)}_{\Gamma}^4 +  2\log^2 \left(T\E^{\mu_0}\pi(y|X;d)\right),
\end{eqnarray*}
where $T = \int_{\mathbb{R}^p} \exp\left(-\frac{1}{2} \norm{y- {\mathcal G}(x;d)}_{\Gamma}^2\right) \dd y = \sqrt{(2\pi)^p\det(\Gamma)}$.
Clearly, $T\pi(y|x;d) \in (0,1)$ and 
since $\log^2(x)$ is a convex function on $(0,1)$, by Jensen inequality, we have
\begin{equation*}
	\log^2 \left(\E^{\mu_0} \left(T \pi(y|X;d)\right) \right) \le  \E^{\mu_0} \log^2 \left(T \pi(y|X;d)\right)
	=  \frac{1}{4}\E^{\mu_0} \norm{y- {\mathcal G}(X; d)}_{\Gamma}^4.
\end{equation*}
Hence, 
\begin{equation} \label{eq:K1}
\begin{aligned}
	K_1 
	& \le \frac{1}{2} \int_{\mathcal{X}} \int_{\mathbb{R}^p}  \norm{y- {\mathcal G}(x;d)}_{\Gamma}^4 \pi(y|x;d) \,\dd y \,\dd \mu_0(x) \\
	& \quad + \frac{1}{2} \int_{\mathcal{X}} \int_{\mathbb{R}^p} \int_{\mathcal{X}} \norm{y- {\mathcal G}(x;d)}_{\Gamma}^4\dd \mu_0(x) \pi(y|\tilde{x}) \,\dd y \,\dd \mu_0(\tilde{x})\\
	& =:   K_{1,1} + K_{1,2}.
\end{aligned}
\end{equation}
For the first integral in \eqref{eq:K1}, we have
\begin{equation}\label{eq:K11}
K_{1,1} = \frac{1}{2}  \int_{\mathcal{X}} \E^{\pi(\cdot|x;d)} \norm{Y- {\mathcal G}(x;d)}_{\Gamma}^4 \,\dd \mu_0(x) \leq C,
\end{equation}
where the constant $C$ depends only on $p$. Now for the second term in \eqref{eq:K1}, by applying Cauchy inequality repeatedly,
\begin{equation} \label{eq:K12}
\begin{aligned}
	K_{1,2} & = \frac{1}{2} \int_{\mathcal{X}} \int_{\mathbb{R}^p} \int_{\mathcal{X}} \norm{y- {\mathcal G}(x;d)}_{\Gamma}^4\dd \mu_0(x) \pi(y|\tilde{x}) \,\dd y \,\dd \mu_0(\tilde{x})\\
	& \le   4 \int_{\mathcal{X}} \int_{\mathcal{X}} \int_{\mathbb{R}^p} \lb( \norm{y- {\mathcal G}(\tilde{x};d)}_{\Gamma}^4 +  \norm{{\mathcal G}(\tilde{x}; d)- {\mathcal G}(x; d)}_{\Gamma}^4 \rb) \pi(y|\tilde{x}) \,\dd y \,\dd \mu_0(\tilde{x}) \,\dd \mu_0(x)\\
	& =  4 \int_{\mathcal{X}} \int_{\mathbb{R}^p} \norm{y- {\mathcal G}(\tilde{x}; d)}_{\Gamma}^4  \pi(y|\tilde{x}) \,\dd y \,\dd \mu_0(\tilde{x})  + 4 \int_{\mathcal{X}} \int_{\mathcal{X}}   \norm{{\mathcal G}(\tilde{x}; d)- {\mathcal G}(x; d)}_{\Gamma}^4 \,\dd \mu_0(\tilde{x}) \,\dd \mu_0(x) \\
	& \le 8 K_{1,1} + 32 \lb( \int_{\mathcal{X}}   \norm{{\mathcal G}(\tilde{x}; d)}_{\Gamma}^4 \,\dd \mu_0(\tilde{x})  + \int_{\mathcal{X}} \norm{{\mathcal G}(x; d)}_{\Gamma}^4  \,\dd \mu_0(x) \rb) \\
	& \le 8C + 64C_G,
\end{aligned}
\end{equation}
thanks to \eqref{eq:K11} and Assumption \ref{ass:2}. Therefore, by combining \eqref{eq:K1}-\eqref{eq:K12}, $K_{1}$ is bounded by some universal constant $C>0$. Similar arguments apply to the term $K_2$. This completes the proof.
\end{proof}

\begin{remark}
Notice that we could replace Assumption \ref{ass:2} by the following conditions:
\begin{itemize}
\item[(B1)] there exists $G\in \R^p$ and $C>0$ such that 
\begin{equation*}
	\sup_{d\in{\mathcal D}}\norm{{\mathcal G}(x,d) - G}_{\Gamma} \leq C \norm{x - \E^{\mu_0}X}_{{\mathcal X}}, \quad
	\sup_{d\in{\mathcal D}}\norm{{\mathcal G}_N(x,d) - G}_{\Gamma} \leq C \norm{x - \E^{\mu_0}X}_{{\mathcal X}}
\end{equation*}
for all $x\in {\mathcal X}$ and $N\in \mathbb{N}$, and
\item[(B2)] the prior $\mu_0$ has a finite fourth order centered moment.
\end{itemize}
In particular, Gaussian prior \rchange{$\mu_0 \sim {\mathcal N}(0, \Sigma)$}{$\mu_0 = {\mathcal N}(0, \Sigma)$} satisfies (B2) and we find that, thanks to Lemma \ref{lem:moments},
\begin{equation*}
    K \le C\lb(1 + \tr(\Sigma) \rb),
\end{equation*}
for some universal constant $C$ depending on $p$, if (B1) is also satisfied. 
\end{remark}

Proposition \ref{prop:Gaussian} leads to the following main result of this section.

\begin{theorem} \label{thm:Gaussian}
Consider the inverse problem \eqref{eq:inverse_example} with an observation operator and surrogate ${\mathcal G}, {\mathcal G}_N : {\mathcal X} \times {\mathcal D}\to \R^p$. Suppose the noise is zero-mean Gaussian, say, $\epsilon \sim \mathcal{N}(0,\Gamma)$. 
Let Assumption \ref{ass:2} hold and assume that 
\begin{equation*}
    \E^{\mu_0} \norm{{\mathcal G}(X;d) - {\mathcal G}_N(X;d)}_{\Gamma}^2 < C\psi(N), \quad \psi(N)\to 0 \text{ as } N\to \infty.
\end{equation*}
Assume further that $\mathcal{G}$ is a continuous function in $d\in \mathcal{D}$. Then there exists $C>0$ such that
\begin{equation*}
    \sup_{d\in \mathcal{D}} |U(d) - U_N(d)| \le C\sqrt{\psi(N)},
\end{equation*}
for all $N$ sufficiently large. Moreover, if $\{d^*_N\}$ is a maximizing sequence of $U_N$ then the limit of any converging subsequence of $\{d^*_N\}$ is a maximizer of $U$.
\end{theorem}

\begin{remark}
Consider the case of linear observation mappings ${\mathcal G}(x; d) = {\mathcal G}(d) x, \; {\mathcal G}_N(x; d) = {\mathcal G}_N(d)x$, and a Gaussian prior \rchange{$\mu_0 \sim {\mathcal N}(0,C_0)$}{$\mu_0 = {\mathcal N}(0,C_0)$}. It follows that
\begin{eqnarray}
	\label{eq:lingaussrem}
	\E^{\mu_0} D_{\mathrm{KL}}(\pi_N(\cdot | X;d) \parallel \pi(\cdot | X;d))  & = & \frac{1}{2} \E^{\mu_0} \norm{({\mathcal G}(d) - {\mathcal G}_N(d))X}_{\Gamma}^2 \nonumber\\
		& = & \norm{C_0^{\frac 12} ({\mathcal G}(d) - {\mathcal G}_N(d)) \Gamma^{-\frac 12}}_{HS}^2,
\end{eqnarray}
where $\norm{\cdot}_{HS}$ stands for the Hilbert--Schmidt norm. Notice that in the case of linear observation map and Gaussian prior, the expected information gain can be explicitly solved. Consequently, the identity \eqref{eq:lingaussrem} can provide intuition regarding the sharpness of Theorem \ref{thm:main2} as demonstrated in the following example.
%
%
\end{remark}


\begin{example}
\label{example:optimality}
Consider the simple inference problem with observation map ${\mathcal G}(x) = ax$, $a>0$, with additive normally distributed noise $\epsilon$. Here, we omit the dependence on $d$. Suppose the prior distribution $\mu_0$ is also normal.
with Gaussian noise $\epsilon \sim \mathcal{N}(0, 1)$ and prior \rchange{$\mu_0 \sim \mathcal{N}(0, 1)$}{$\mu_0 = \mathcal{N}(0, 1)$}. Moreover, suppose we approximate ${\mathcal G}(x)$ with a surrogate $\mathcal{G}_N(x) = a_Nx$ for some $a \leq a_N \leq C$.  
A straightforward calculation will now yield
\begin{equation*}
    |U_N - U| = \frac{1}{2}\lb|\log \frac{a_N^2+1}{a^2+1}\rb|.
\end{equation*}
Now since
$1 - \frac{1}{x} \le \log (x) \le x - 1$, 
for all $x>0$, we deduce that
\begin{equation} \label{eq:ex2}
    \frac{|a_N-a|(a_N+a)}{2(a_N^2+1)} \le |U_N - U| \le \frac{|a_N-a|(a_N+a)}{2(a^2+1)}.
\end{equation}
We observe from \eqref{eq:ex2} that
the best possible convergence rate is $|a_N - a|$ and, indeed, by Lemma \ref{lem:KL_Gaussian} we have
\begin{equation*}
    \E^{\mu_0} D_{\mathrm{KL}} \lb(\pi_N(\cdot\; |\; X) \parallel \pi(\cdot \; | \; X) \rb) = \frac 12 |a_N - a|^2.
\end{equation*}
In consequence, the convergence rate of Theorem \ref{thm:main1} is asymptotically sharp.
\end{example}

\section{Numerical Simulations}
\label{sec:numeric}

Let us numerically demonstrate the convergence rates predicted by Theorems \ref{thm:main1} and \ref{thm:main2} with three examples. Note that these examples were also featured in \citep{huan2013simulation, huan2014gradient} for numerical demonstrations.

\subsection{Piecewise linear interpolation in one dimension \label{1D Problem}}

Consider a measurement model
\begin{equation}
y(x, d) = \mathcal{G}(x, d)+\eta \in \R^2,
\end{equation}
for $x\in {\mathcal X} = [0,1]$ and $d\in {\mathcal D} = [0,1]^2$, where $\eta \sim \mathcal{N}\left(0,10^{-4}I\right)$ and
\begin{equation}
\label{eq:exampleG}
{\mathcal G}(x, d) =\left[\begin{array}{c}
x^{3} d_{1}^{2}+x \exp \left(-\left|0.2-d_{1}\right|\right) \\
x^{3} d_{2}^{2}+x \exp \left(-\left|0.2-d_{2}\right|\right)
\end{array}\right].
\end{equation}
As the prior distribution $\mu_0$ we assume a uniform distribution on the unit interval\rnew{, that is $\mu_0 = \mathcal{U}(0, 1)$}.

Here we consider a surrogate model $\mathcal{G}_{N}$ obtained by piecewise linear interpolation with respect to $x$. \rnew{More precisely, the expression of the surrogate model is given by 
\begin{equation}
\mathcal{G}_{N}(x,d) = \frac{x_i - x}{h_i}\mathcal{G}(x_{i-1},d) + \frac{x - x_{i-1}}{h_i}\mathcal{G}(x_{i},d), \quad x\in [x_{i-1},x_i], 
\end{equation}
where $h_i=x_i - x_{i-1}$, and $i = 1,\cdots , N $.\\
}
It is well known that the interpolation of $f\in H^{2}(0,1)$ on equidistant nodes $x_{0} = 0 < x_{1} < x_{2} < \cdots < x_{N} = 1$ satisfies 
$\norm{f-f_N}_{L^2(0,1)} \leq C N^{-2} \norm{f''}_{L^2(0,1)}$, see e.g. \cite{han2009theoretical}. Also, notice carefully that for $d_1 = d_2 = 0$ we have ${\mathcal G}$ is linear and the approximation is accurate.
In consequence, we have
\begin{equation}
	\label{eq:pli_1d_aux1}
	\sup_{d\in{\mathcal D}} \E^{\mu_0} \norm{{\mathcal G}(X; d)-{\mathcal G}_N(X; d)}^2
	= \sup_{d\in{\mathcal D}} \int_0^1 \norm{{\mathcal G}(x; d)-{\mathcal G}_N(x; d)}^2dx \leq C N^{-4} .
\end{equation}
Moreover, it is straightforward to see that the mapping $x \mapsto \mathcal{G}(x;d)$ is bounded on the interval $[0,1]$ uniformly in $d$ and, therefore, satisfies the Assumption \ref{ass:2}. 

We have numerically evaluated both the uniform error $E_N := \sup_{d\in {\mathcal D}} |U(d)-U_N(d)|$ and the left-hand side term in inequality \eqref{eq:pli_1d_aux1} with varying $N$.
For evaluating $U(d)$ and $U_N(d)$ we use trapezoidal rule for discretizing $\mathcal{X}$  \rchange{with a}{using an} equidistant grid with 251 nodes. For the data space $\R^2$ we utilized Gauss--Hermite--Smolyak quadrature with 3843 nodes.
For estimating $E_N$ and the left-hand side of \eqref{eq:pli_1d_aux1} we fix a $21\times 21$ grid for the design space ${\mathcal D}$ over which we optimize. Moreover, for any $d$ we numerically solve $\E^{\mu_0} \norm{{\mathcal G}(X; d) - {\mathcal G}_N(X;d)}^2$ using \rnew{midpoint rule for the expectation with} an equidistant grid in $\mathcal{X}$ with 1001 nodes. \rnew{The supreme norm is approximated by calculating the maximum value of the evaluations in the grid of $\mathcal{D}$}.


\begin{figure}[htp!]
\centering\includegraphics[scale = 0.4]{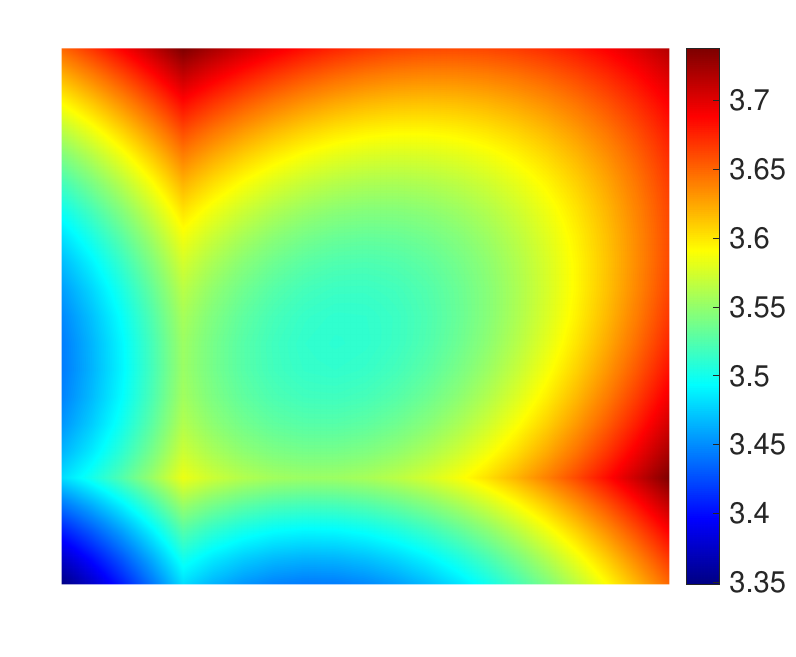}
\centering\includegraphics[scale = 0.4]{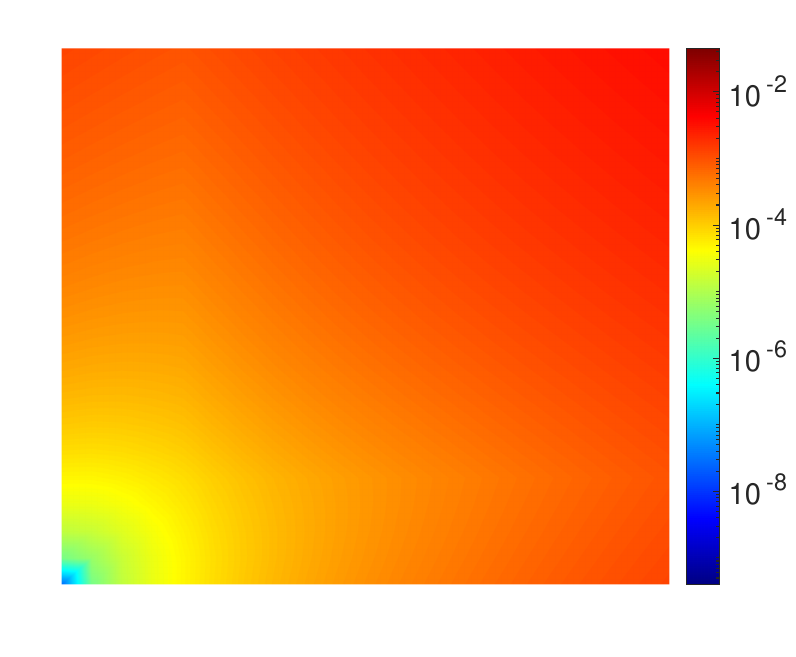}
\centering\includegraphics[scale = 0.4]{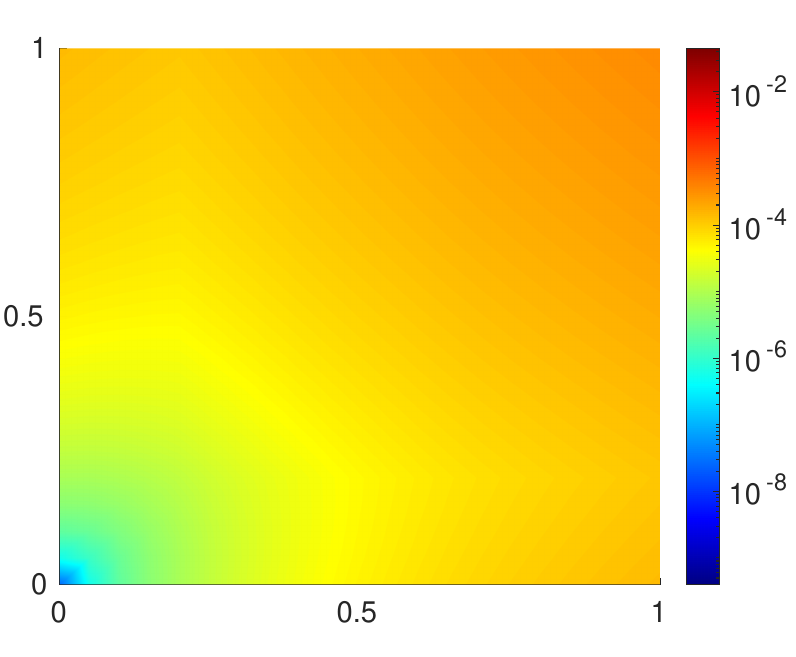}
\centering\includegraphics[scale = 0.4]{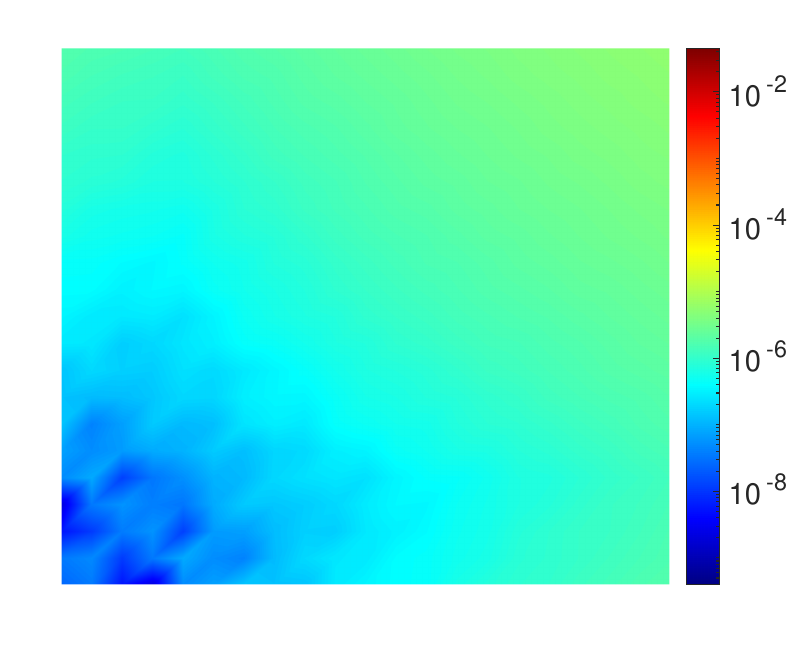}
\caption{Convergence of the expected utility for example \ref{1D Problem} is illustrated \rnew{(evaluating in MATLAB using a grid of $21 \times 21$ for the design space $\mathcal{D}$)}. The true utility $U(d)$ is plotted w.r.t. $d\in {\mathcal D} = [0,1]^2$ on the left upper corner. The approximation error $|U(d)-U_N(d)|$ is plotted with $N = 9$ (right upper corner), $33$ (left bottom corner) and $257$ (right bottom corner). Smaller error towards the origin $d_1=d_2=0$ is due to the linearity of ${\mathcal G}$ at this point.}
\label{Fig_2_1}
\end{figure}


In Figure \ref{Fig_2_1} we have plotted the expected information gain $U$ and the errors $|U_N - U|$ with values $N=9, 33$ and $257$.
The errors $E_N$ and $\sup_{d\in{\mathcal D}} \E^{\mu_0} \norm{{\mathcal G}(X; d)-{\mathcal G}_N(X; d)}^2$ are plotted in Figure \ref{Fig_3} for values varying between $N=2$ and $N=10^3$.
Moreover, we have also added the theoretical upper bound ${\mathcal O}(N^{-2})$.
We observe that the quantities have the same asymptotic behaviour following the theoretical bound.  
 
\begin{figure}[htp!]
\centering\includegraphics[scale = 0.7]{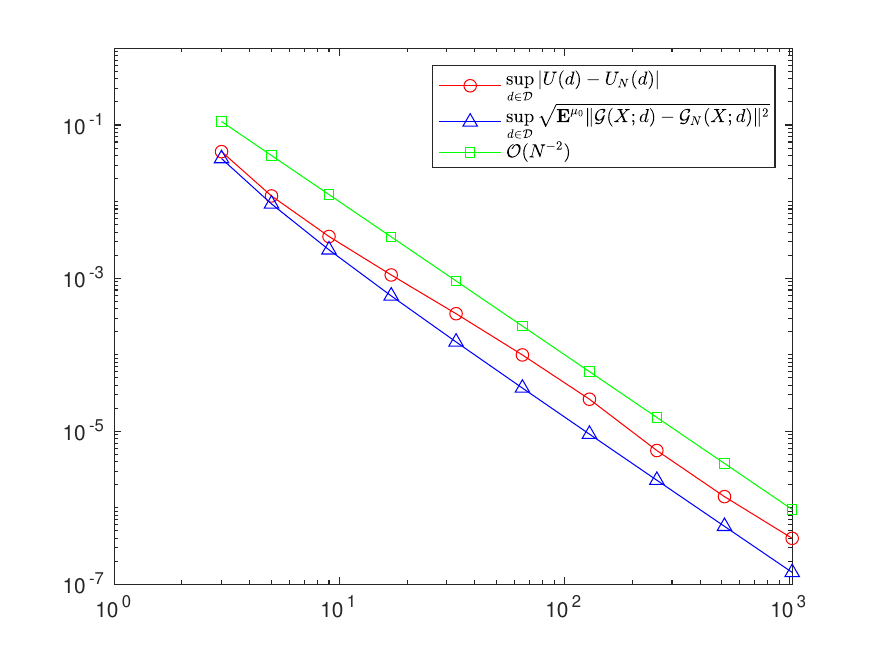}
\caption{
The convergence rate predicted by Theorem \ref{thm:main1} is demonstrated for example \ref{1D Problem}.
We plot the uniform errors of the expected utility (red curve) and uniform $L^2$-distance of the observation map and the surrogate with respect to the prior distribution (blue curve) with varying $N$.
For reference, a line proportional to $N^{-2}$ is plotted.}
\label{Fig_3}
\end{figure}


\subsection{Sparse piecewise linear interpolation in three dimensions}
\label{3D Problem}

Consider the observation mapping ${\mathcal G} : [0,1] \times {\mathcal D} \to \R^2$ for ${\mathcal D}=[0.2,1]^2$ defined by formula \eqref{eq:exampleG}. In this subsection, we formulate a surrogate model ${\mathcal G}_N$ by interpolating data in both $x$ and $d$ variables. We apply a piecewise linear interpolation using a sparse grid with Clenshaw--Curtis configuration \citep{le2010spectral}.
Since ${\mathcal G}$ has continuous second partial derivatives on its domain, the error of this interpolation can be bounded by
\begin{equation}
\Vert \mathcal{G} - \mathcal{G}_{N}\Vert_{L^\infty(\mathcal{X} \times \mathcal{D})} = \mathcal{O}(N^{-2}(\log N)^{3(n-1)}),
\end{equation}
where $n=3$ stands for the dimension of the domain,
see e.g.  \citep{barthelmann2000high, novak1996high, bungartz1998finite}. \rnew{For prior, we also assume here that the prior measure is a uniform distribution, $\mu_0 = \mathcal{U}(0, 1)$.}
Following \eqref{eq:pli_1d_aux1} we immediately observe 
\begin{equation}
    \label{eq:pli_3d_aux1}
	\sup_{d\in{\mathcal D}} \sqrt{\E^{\mu_0} \norm{{\mathcal G}(X; d) - {\mathcal G}_N(X;d)}^2} = \mathcal{O}(N^{-2}(\log N)^6).
\end{equation}
Similar to section \ref{1D Problem} the mappings ${\mathcal G}$ and ${\mathcal G}_N$ are bounded and \rchange{satisfies}{satisfy} Assumption \ref{ass:2}.

\begin{figure}[htp!]
\centering\includegraphics[scale = 0.4]{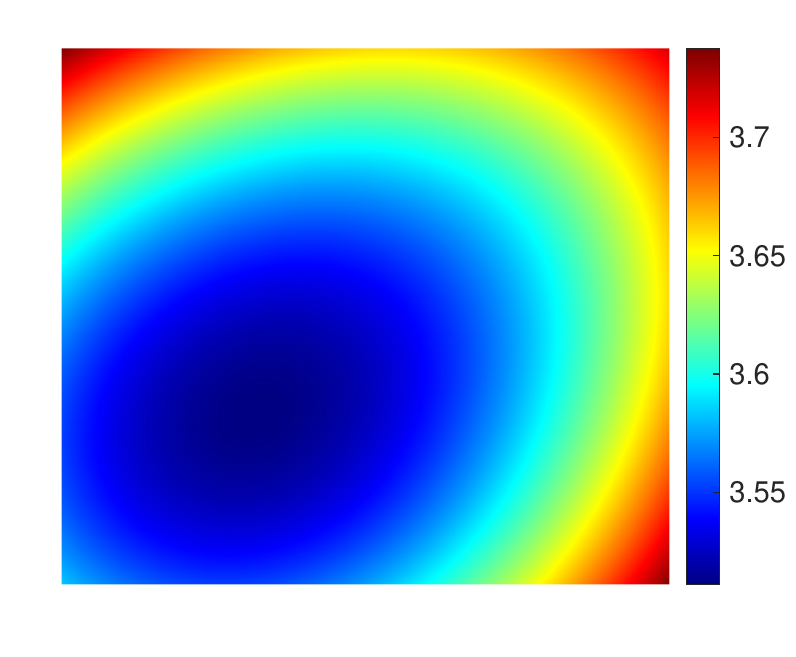}
\centering\includegraphics[scale = 0.4]{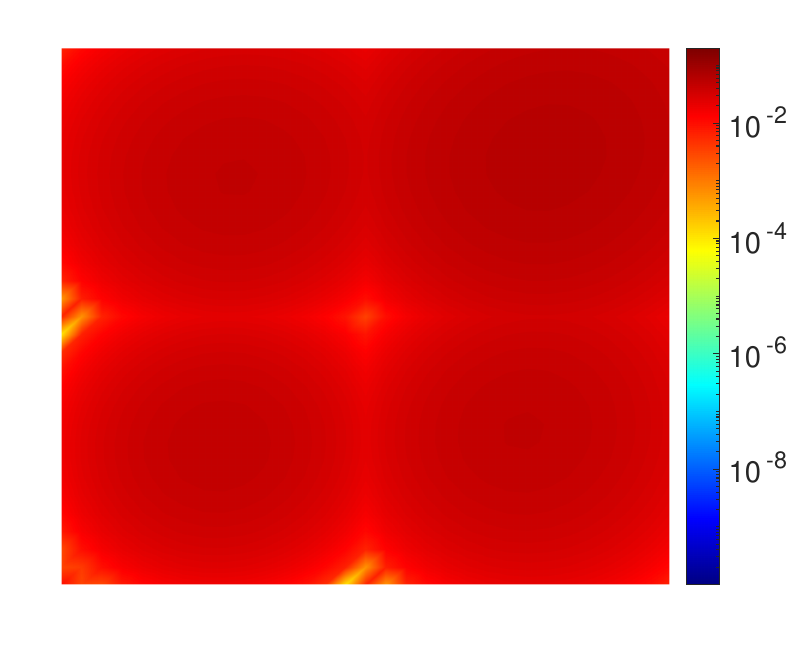}
\centering\includegraphics[scale = 0.4]{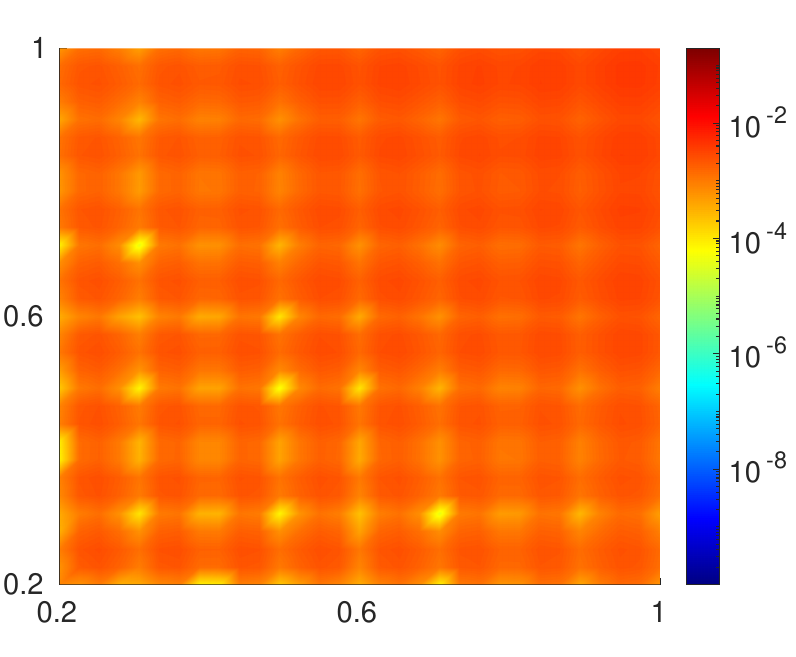}
\centering\includegraphics[scale = 0.4]{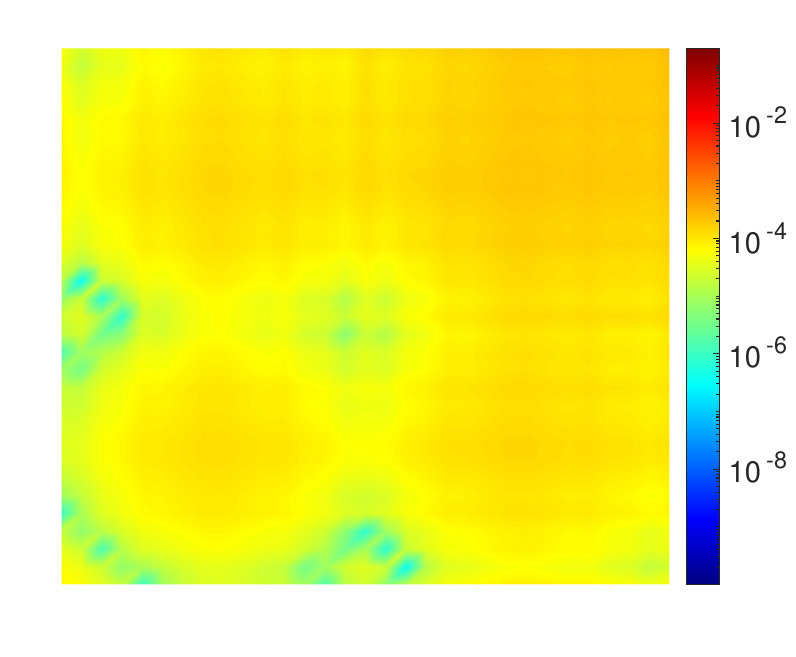}
\caption{Convergence of the expected utility for example \ref{3D Problem} is illustrated. The true utility $U(d)$ is plotted w.r.t. $d\in {\mathcal D} = [0.2,1]^2$ on the left upper corner. The approximation error $|U(d)-U_N(d)|$ is plotted with $N = 18$ (right upper corner), $108$ (left bottom corner) and $632$ (right bottom corner). Smaller error towards the origin $d_1=d_2=0$ is due to the linearity of ${\mathcal G}$ at this point.}
\label{Fig_5}
\end{figure}

We again evaluate the difference $E_N := \sup_{d\in {\mathcal D}} |U(d)-U_N(d)|$ and 
the left-hand side term in inequality \eqref{eq:pli_3d_aux1} with varying $N$.
In the same manner as the previous example, we estimated the expected information gains $U(d)$ and $U_N(d)$ using quadrature rules. For the integral over the space $\mathcal{X}$ we used again trapezoidal rule with 251 equidistant points, while for the space $\mathcal{Y}$ the Gauss--Hermite--Smolyak quadrature with 1730 nodes. 
For estimating $E_N$ and the left-hand side of \eqref{eq:pli_1d_aux1} we fix a $31\times 31$ grid for the design space ${\mathcal D}$ over which we optimize.
For constructing the surrogate we utilized the algorithm and toolbox detailed in \citep{klimke2005algorithm, klimke2006sparse}. \rchange{We use a grid of $1001\times 31 \times 31$ nodes on $\mathcal{X}\times\mathcal{D}$ to evaluate the left-hand side of (\ref{eq:pli_3d_aux1}).}{The supreme norm and the expectation with respect to the prior measure are estimated in the same way as the first example, the left-hand side of (\ref{eq:pli_3d_aux1}) is evaluated with a grid of $1001\times 31 \times 31$ nodes on $\mathcal{X}\times\mathcal{D}$}. 



\begin{figure}[htp!]
\centering\includegraphics[scale = 0.7]{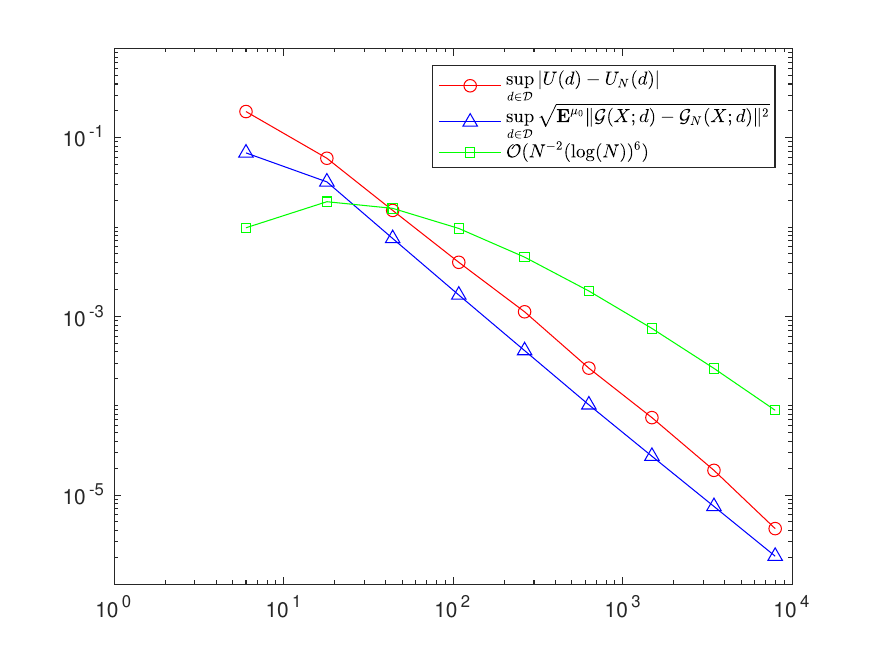}
\caption{The convergence rate predicted by Theorem \ref{thm:main1} is demonstrated for example \ref{3D Problem}.
We plot the uniform errors of the expected utility (red curve) and the uniform $L^2$-distance of the observation map ${\mathcal G}$ and the surrogate ${\mathcal G}_N$ with respect to the prior distribution (blue curve) with varying $N$.
For reference, a curve proportional to $N^{-2}(\log N)^6$ is plotted.
}
\label{Fig_6}
\end{figure}

Figure \ref{Fig_5} plots the expected information gain $U$ and the errors $|U_N - U|$ with values $N=18, 108$ and $632$. 
Again, the errors $E_N$ and $\sup_{d\in{\mathcal D}} \E^{\mu_0} \norm{{\mathcal G}(X; d)-{\mathcal G}_N(X; d)}^2$ are plotted in Figure \ref{Fig_6} for values varying between $N=5$ and $N=10^4$.
We observe that the numerical convergence rates of the error terms are of the same order while the theoretical upper bound also seems to align asymptotically with these rates asymptotically. After $N=10^5$ the convergence rates saturate around the value $10^{-7}$ due to limited numerical precision of our implementation.

\newpage
\subsection{Optimal sensor placement for an inverse heat equation\label{4D Problem}}

Consider the heat equation on the domain $\Omega \times {\mathcal T} = [0,1]^2 \times [0,0.4]$ with a source term $S$ with zero Neumann boundary condition and initial value according to
\begin{align*}
    \frac{\partial v}{\partial t} - \Delta v &= S(\cdot,x), && (z,t) \in (0,1)^2 \times \mathcal{T} \\
    \nabla v \cdot n &= 0, && (z,t) \in \partial\Omega \times \mathcal{T} \\
    v(z,0) &= 0, && z \in \Omega
\end{align*}
where $n$ is a boundary normal vector.
We assume that the source term is given by
\begin{equation*}
S(z,t,x)=
\begin{cases}
\frac{s}{2 \pi h^2} \exp \left(-\frac{\left\|z-x\right\|^2}{2 h^2}\right), & 0 \leq t<\tau, \\
0, & t \geq \tau.
\end{cases}
\end{equation*}
with parameter values $s=2$, $h=0.05$ and $\tau = 0.3$. Moreover, the parameter $x$ is interpreted as the position of the source. We assume that we can observe $v$ at a location $d\in {\mathcal D} := [0.1, 0.9]^2 \subset \Omega$ at predefined times $t_i$, $i=1,...,5$. The inverse problem in this setting is to estimate the source location $x \in {\mathcal X} := \Omega$ given the data $y = \{v(d, t_i)\}_{i=1}^5 \in \R^5$, i.e., invert the mapping
\begin{equation*}
	{\mathcal G} : \Omega \times {\mathcal D} \to \R^5, \quad (x,d) \mapsto y.
\end{equation*}

Here, we consider a Bayesian design problem with the aim to optimize the measurement location $d$ given a uniform prior \rchange{on}{of} the source location $x$ \rnew{on $(0, 1)^2$} and an additive Gaussian noise $\eta \sim \mathcal{N}(0,0.01 I)$ in the measurement.


Numerical implementation of ${\mathcal G}$ was carried out with a finite difference discretization for the spatial grid, while a fourth order backward differentiation was used for the temporal discretization. 
The surrogate observation mapping ${\mathcal G}_N$ is obtained by polynomial chaos expansion with Legendre polynomials. Here, we implemented the projection on an extended domain $\Omega \times \Omega$ instead of $\Omega \times {\mathcal D}$ to avoid any potential boundary issues. The implementation follows \citep{huan2014gradient} and more details can be found therein. In short, we define ${\mathcal G}_N$ as a sum of polynomials
\begin{equation*}
	{\mathcal G}_N(x,d) = \sum_{{\bf j} = 0}^N {\mathcal G}_{{\bf j}} \Psi_{{\bf j}}(x,d),
\end{equation*}
where ${\bf j} \in {\mathbb N}^4$ and $\{\Psi_{{\bf j}} \}_{{\bf j}\in \mathbb{N}^{4}}$ are a Legendre polynomial basis in $\Omega\times \Omega$ with the standard inner product, i.e. the $L^2$-inner product weighted by the uniform prior. Moreover, the coefficients satisfy
\begin{equation}
	\label{eq:PC_coef}
	{\mathcal G}_{{\bf j}}  = \frac{\int_{\Omega\times\Omega} {\mathcal G}(x; d) \Psi_{{\bf j}}(x, d) \dd x \dd y }{\int_{\Omega\times \Omega} \Psi_{{\bf j}}^2(x,d) \dd x\dd y}.
\end{equation}
The parameter $N$ is the truncation level of the polynomial chaos. Notice carefully that we have included the design parameter $d$ in the approximation.



We compute the coefficients ${\mathcal G}_{{\bf j}}$ with a Gauss--Legendre--Smolyak quadrature with Clenshaw-Curtis configuration with a high number (of the order $10^7$) of grid points and assume in the following that the truncation level $N$ is the dominating factor for the surrogate error $\sup_{d\in \Omega} \norm{{\mathcal G}- {\mathcal G}_N}$.
 
The utility functions were estimated as follows: for the integral over the data space $\R^5$ we used the Gauss--Hermite--Smolyak quadrature with $117$ nodes, while on the domain $\Omega$ we implemented a bidimensional Clenshaw--Curtis--Smolyak quadrature with $7682$ nodes. For the evaluation of maximal difference\rnew{ of the expected utility}, we fixed an equidistant \rold{a} grid with $40\times 40$ nodes in $\mathcal{D}$. \rchange{For evaluating $\sup_{d\in{\mathcal D}} \E^{\mu_0}\norm{{\mathcal G}(X; d) - {\mathcal G}_N(X; d)}^2$ we utilized a grid of $25$ nodes in each direction on $\mathcal{X}\times\mathcal{D}$}{For evaluating the expectation in the term $\sup_{d\in{\mathcal D}} \E^{\mu_0}\norm{{\mathcal G}(X; d) - {\mathcal G}_N(X; d)}^2$ we use midpoint rule and a grid of $25$ nodes in each direction on $\mathcal{X}$. The supreme norm is approximated by calculating the maximum element of $25\times 25$ nodes in a grid of $\mathcal{D}$}. 


\begin{figure}[htp!]
\centering\includegraphics[scale = 0.4]{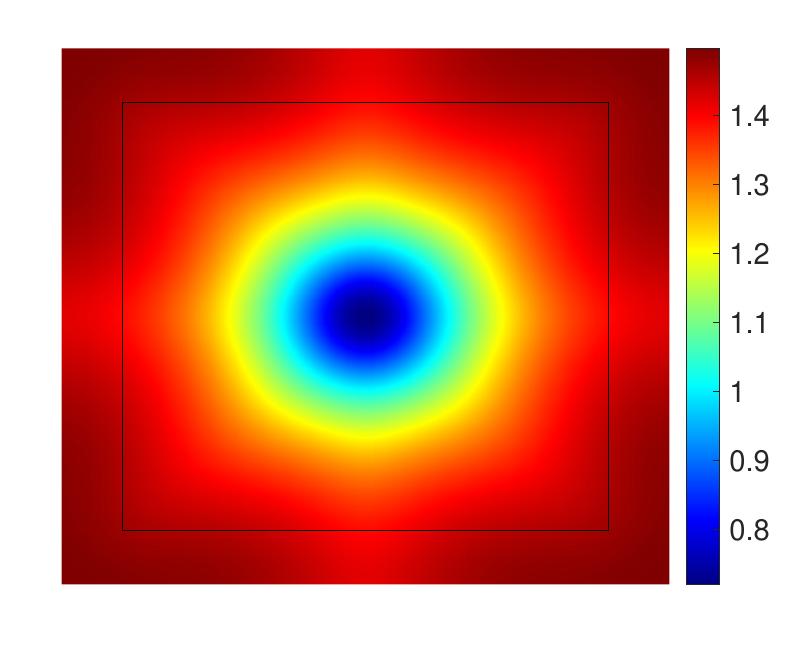}
\centering\includegraphics[scale = 0.4]{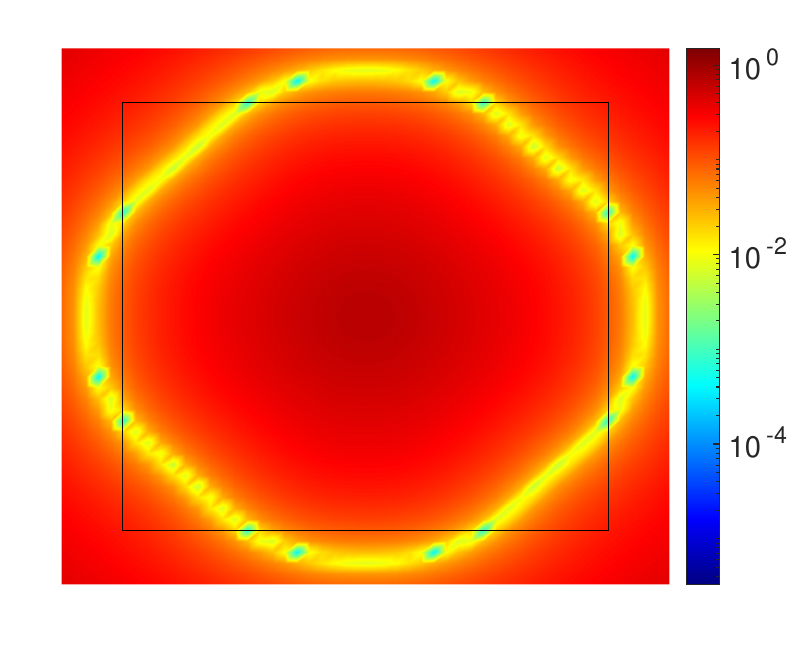}
\centering\includegraphics[scale = 0.4]{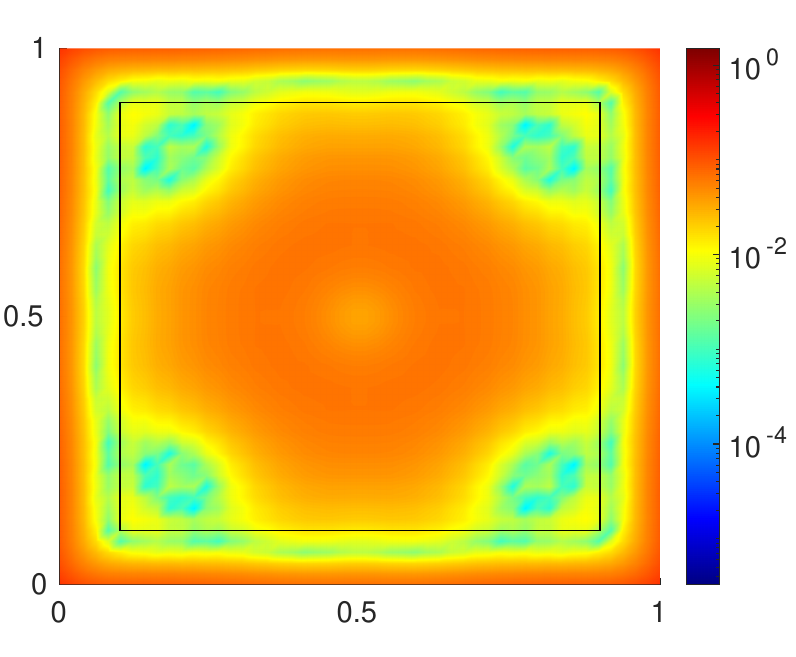}
\centering\includegraphics[scale = 0.4]{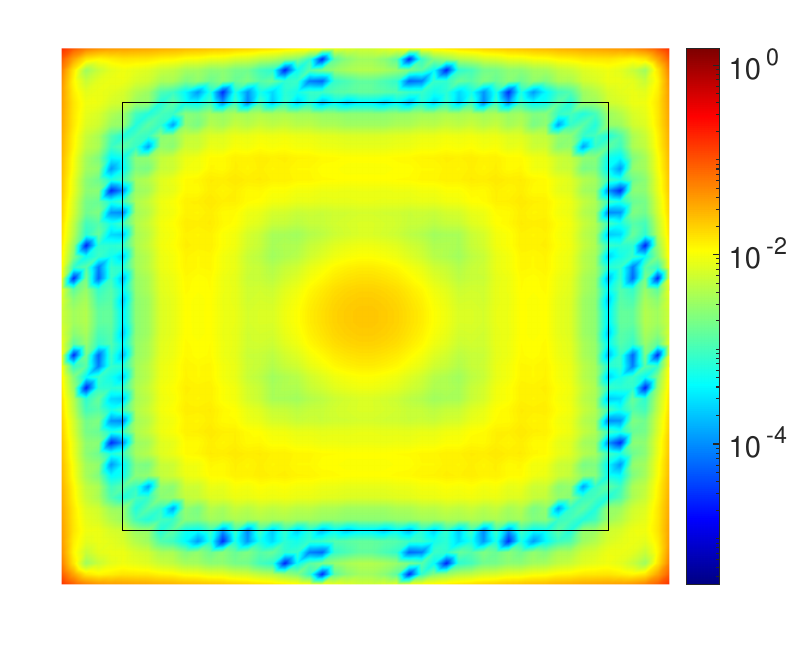}
\caption{Convergence of the expected utility for Example \ref{4D Problem} is illustrated.  The true utility $U(d)$ is plotted on the left upper corner. The approximation error $|U(d)-U_N(d)|$ is plotted with $N = 2$ (right upper corner), $6$ (left bottom corner) and $14$ (right bottom corner). The design space ${\mathcal D}$ is highlighted with a black box.}
\label{Fig_7}
\end{figure}


\begin{figure}[htp!]
\centering\includegraphics[scale = 0.7]{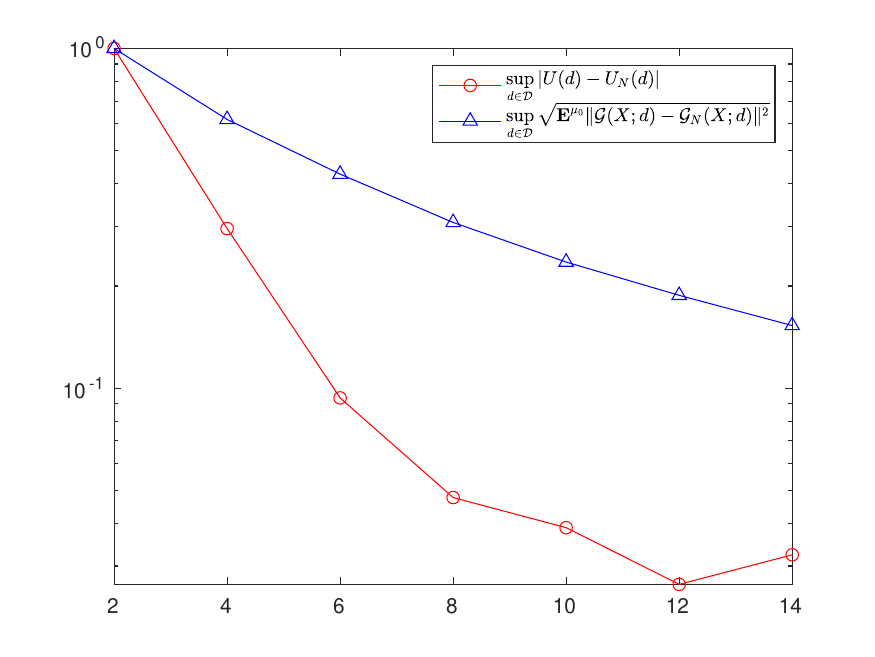}
\caption{
The convergence rate predicted by Theorem \ref{thm:main1} is demonstrated for example \ref{4D Problem}.
We plot the uniform errors of the expected utility (red curve) and the uniform $L^2$-distance of the observation map ${\mathcal G}$ and the polynomial chaos surrogate ${\mathcal G}_N$ with respect to the prior distribution (blue curve) with varying $N$.}
\label{Fig_8}
\end{figure}

In Figure \ref{Fig_7} we show the approximation of the utility functions for every polynomial degree. In both \rnew{cases} we can see a visual convergence with respect to the polynomial degree.
In the spirit of previous examples, Figure \ref{Fig_8} contains the errors between the surrogate models and the utility functions. 
Compared to the previous examples, the complexity of this computational task is substantially larger inducing larger numerical errors especially in the evaluation of $E_N$. Also, as is seen from Figure \ref{Fig_7}, the error close to the boundary of $\Omega$ is converging substantially slower. However, as illustrated by Figure \ref{Fig_8} the two rates of convergence are aligned with ${\mathcal D}$.


\section{Conclusion}
We have developed a framework to study the stability of the expected information gain for optimal experimental designs in infinite-dimensional Bayesian inverse problems. We showed a uniform convergence for the expected information gain, with a sharp rate, given approximations of the likelihood. In the case of Bayesian inverse problems with Gaussian noise, this rate is proved to coincide with the $L^2$-convergence of the observation maps with respect to the prior distribution. Moreover, we also showed that the optimal design variable is also stable in the approximation. The results are illustrated by three numerical experiments.

Possible extensions of this work naturally include considering the stability of various other utilities such as the negative least square loss \citep{CV95} or utilities related to Bayesian optimization \citep{shahriari2015taking}. Here, we only considered perturbations of the utility induced by a surrogate likelihood model. However, a further numerical (such as Monte Carlo based) approximation for the expected utility is most of the time needed. A number of results considering convergence with respect to Monte Carlo error for fixed design have been provided (see e.g. \cite{RDMP16} and references therein) while the uniform case has not been addressed to the best of our knowledge. 
Finally, $\Gamma$-convergence does not directly provide a convergence rate for the optimal designs, which remains an interesting open problem.


%
%

\begin{acks}[Acknowledgments]
The work of TH and DLD was supported by the Academy of Finland (decision numbers 326961 and 345720). 
JRRG was supported by the LUT Doctoral School.
\end{acks}
\bibliographystyle{imsart-nameyear} 
\bibliography{bibliography}       


\end{document}